\documentclass[reqno]{amsart}
\usepackage{a4wide,calrsfs}
\usepackage{amsmath,bbm}
\usepackage{mathtools}
\usepackage{color}
\usepackage{epsfig}
\usepackage{subfigure}
\usepackage{amssymb}
\usepackage{enumerate}
\usepackage{graphicx}
\usepackage{subfigure}
\usepackage{rotating}
\usepackage{stmaryrd}
\usepackage{upgreek}
\usepackage{amsthm}
\usepackage{amsopn}
\usepackage{epstopdf}
\usepackage{cite}
\DeclareMathAlphabet{\mathcal}{OMS}{cmsy}{m}{n}

\newtheorem{theorem}{Theorem}[section]
\newtheorem{proposition}{Proposition}[section]

\newcommand{\R}{\mathbb{R}}

\newcommand{\bu}{\mathbf{u}}
\newcommand{\bw}{\mathbf{w}}

\newcommand{\bC}{\mathbf{C}}
\newcommand{\bF}{\mathbf{F}}

\newcommand{\bg}{\mathbf{g}}

\newcommand{\bL}{\mathbf{L}}
\newcommand{\bV}{\mathbf{V}}

\newcommand{\bx}{\mathbf{x}}

\newcommand{\bv}{\mathbf{v}}

\newcommand{\balpha}{\boldsymbol{\alpha}}

\newcommand{\diver}{\mathrm{div}\,}

\allowdisplaybreaks
\numberwithin{equation}{section}

\title[Existence and uniqueness of bioconvective flow model]
{A result on the existence and uniqueness of stationary solutions
for a bioconvective flow model}

\author[A. Coronel]{An{\'\i}bal\ Coronel$^\dag$}
\author[L. Friz]{Luis Friz$^\dag$}
\author[I. Hess]{Ian Hess$^\dag$}
\author[A. Tello]{Alex Tello$^\dag$}

\thanks{$^\dag$
GMA, Departamento de Ciencias B\'asicas, Facultad de Ciencias,
Universidad del B\'{\i}o-B\'{\i}o, Campus Fernando May, Chill\'{a}n, Chile.}
\email{acoronel@ubiobio.cl,lfriz@ubiobio.cl,ihess@alumnos.ubiobio.cl,alextello21@gmail.com}

\date{\today}

\begin{document}

\begin{abstract}
In this note we prove the existence and uniqueness of weak solutions 
for the boundary value problem modelling the stationary case of 
the bioconvective flow problem introduced
by Tuval et. al. (2005, {\it PNAS} 102, 2277--2282).
We derive some appropriate a priori estimates for the weak solution,
which implies 
the existence, by application of  Gossez theorem, and the uniqueness
by standard methodology of comparison of two arbitrary solutions.

\end{abstract}
\maketitle

\numberwithin{equation}{section}
\newtheorem{thm}{Theorem}[section]
\newtheorem{lem}[thm]{Lemma}
\newtheorem{prop}[thm]{Proposition}
\newtheorem{cor}[thm]{Corollary}
\newtheorem{defn}{Definition}[section]
\newtheorem{conj}{Conjecture}[section]
\newtheorem{exam}{Example}[section]
\newtheorem{rem}{Remark}[section]
\allowdisplaybreaks


\section{Introduction}

The bioconvection is an important process in the biological treatment and in the life
of some microorganisms. In a broad sense, the biconvection is originated by 
the concentration of upward swimming microorganisms in a culture fluid.  
It is well known that, under some
physical assumptions, the process can be described by a 
mathematical models which are called bioconvective flow models.
The first model of this kind was derived by Y. Moribe \cite{moribe_1973}
and independently by M. Levandodovsky, W. S. Hunter and
E. A. Spiegel \cite{levandowsky_1975} (see also \cite{kanon} for the mathematical
analysis). In that models the unknowns are the velocity of the fluid,
the pressure of the fluid and the local concentration of microorganisms. 
More recently, Tuval et. al \cite{tuval_2005} have bee introduced 
a new bioconvective flow model considering also as an unknown variable the oxygen 
concentration. Some advances in mathematical analysis and some numerical results
of this new model are presented
in \cite{liu_2011} and \cite{lee_2015}, respectively.

In this note, we are interested with the existence and uniqueness of solutions
for the stationary problem associated to
bioconvective system given in~\cite{tuval_2005}
when the physical domain is a three-dimensional chamber~\cite{lee_2015}
(a parallelepiped).
Thus,
the stationary bioconvective flow problem to be analyzed is formulated 
as follows.
Given the external force $\mathbf{F}$,
the source functions $f_n,f_c$ and the dimensionless function 
$r$ find the velocity of the fluid $\mathbf{u}=(u_1,u_2,u_3)^t$, 
the fluid pressure $p$,
the local concentration of bacteria $n$ and  the local
concentration of oxygen $c$ satisfying the boundary value problem
\begin{align}
-S_c \Delta\bu+(\bu\cdot\nabla)\bu+S_c \nabla p&=\gamma S_c n\bg
+\mathbf{F},
&& \mbox{in $\Omega:=\prod_{i=1}^3[0,L_i]$},
\label{eq:navstokes}\\
\diver (\bu)&=0,
&&  \mbox{in $\Omega$},
\label{eq:diver}\\
-\Delta n+(\bu\cdot\nabla)n+\chi\,\diver(n\,
r(c)\nabla c)&=f_n,
&& 
\mbox{in $\Omega$},
\label{eq:bacter}\\
 -\delta \Delta c+(\bu\cdot\nabla)c+\beta \, r(c)n&=f_c,
 && 
 \mbox{in $\Omega$},
 \label{eq:oxigeno}\\
\nabla c\cdot \boldsymbol{\nu}
=\nabla n\cdot \boldsymbol{\nu}&=0,\;\; \bu=0,
&& 
 \mbox{on $\partial\Omega_L$ ($x_3=0$)},
\label{eq:fronterainf}
\\
\chi\, n\, r(c) \nabla c\cdot \boldsymbol{\nu}
-\nabla n\cdot \boldsymbol{\nu}
&=0,\;\; \bu=0,
&& 
 \mbox{on $\partial\Omega_U:=\partial\Omega-\partial\Omega_L$.}
 \label{eq:fronterasup}
\end{align}
Here $ \boldsymbol{\nu}$ is the unit external normal
to $\partial\Omega$;
$\mathbf{g}=(0,0,-g)$ is the gravity with $g$
the acceleration of gravity constant; and $S_c,\gamma,\alpha,\delta$
and $\beta$ are some physical parameters defined as follows
\begin{align*}
S_c=\frac{\eta}{D_n\rho},\;\;
\gamma=\frac{V_bn_r(\rho_b-\rho)L^3}{\eta D_n},\;\;
\chi=\frac{\overline{\chi} c_{air}}{D_n},\;\;
\delta=\frac{D_c}{D_n},\;\;
\beta=\frac{kn_rL^2}{c_{air}D_n},
\end{align*}
 with $\eta$ the fluid viscosity,
 $D_n$ the diffusion constant for bacteria,
 $D_c$ the diffusion constant for oxygen,
 $\rho$
the fluid density,
 $\rho_b$ the bacterial density, $V_b>0$ the bacterial volume,
$n_r$ a characteristic cell density, $L$ a characteristic length,
$\overline{\chi} $ the chemotactic sensitivity, $c_{air}$
the oxygen concentration above the fluid and $k$
is the oxygen consumption rate.

We consider the standard notation of the Lebesgue and Sobolev spaces
which are used in the analysis of Navier-Stokes and related equations
of fluid mechanics, see \cite{adams_book,boyer_book,ene_book,ladyzhenskaya_book,temam_1977}
for details of
specific definitions. 
In particular, 
we use the following rather common spaces notation
\begin{align*}
&H^{m}(\Omega)=W^{m,2}(\Omega),
\quad
\tilde{H}^1(\Omega)=\left\{f\in H^1(\Omega):\int_{\Omega}fd\bx=0\right\},
\quad
H^{1}_0(\Omega)=\overline{\bC^\infty_{0} (\Omega)}^{\|\cdot\|_{H^1(\Omega)}},
\\
&\bC^\infty_{0,\sigma} (\Omega)=
 \Big\{ \bv \in (C^\infty_0
(\Omega))^3 : \ \operatorname{div}(\bv) = 0  \Big\},
\quad
\bV=\overline{\bC^\infty_{0,\sigma} (\Omega)}^{\|\cdot\|_{H^1_0(\Omega)}},
\end{align*}
where $\overline{A}^{\|\cdot\|_{B}}$ 
denotes the completation of $A$
in $B$. Also ,we consider the notation for the applications
$a_0:\bV \times \bV\to\R,$ $a:H^1(\Omega) \times H^1(\Omega)\to\R,$
$b_0:\bV \times \bV \times \bV\to\R$ and
$b:\bV \times H^1(\Omega) \times H^1(\Omega) \to\R$, which are 
defined as follows
\begin{align*}
a_0(\bu, \bv) =(\nabla\bu,\nabla\bv),
\;\;
a(\phi, \psi) =(\nabla\phi,\nabla\psi),
\;\;
b_0(\bu, \bv, \bw) = ((\bu\cdot\nabla) \bv, \bw),
\;\;
b (\bu, \phi, \psi) = (\bu\cdot \nabla \phi,\psi),
\end{align*}
where $(\cdot,\cdot)$ is the standard inner product in $L^2(\Omega)$
or  $\bL^2(\Omega)$.
It is well known that  $a_0$ and $a$ are bilinear
coercive forms,  $b_0$ and $b$
are well defined trilinear forms with the following properties:
\begin{align}
b_0 (\bu, \bv, \bw) = -b_0 (\bu, \bw, \bv), 
\quad
b (\bu, \phi, \psi)= -b (\bu, \psi, \phi),
\quad
b_0 (\bu, \bv, \bv)&=0, &b (\bu, \phi, \phi)&=0,
\label{eq:prop_B0_B1:2}
\end{align}
for all $\bu, \bv, \bw\in \bV$ and $\psi, \phi\in  H^1(\Omega).$
Moreover, we need to introduce some  notation related with
some useful Sobolev inequalities and estimates
for $b$ and $b_0$.
There exists $C_{poi}>0,$ $C_{tr}>0$ and $C_1$ depending only on $\Omega$ such that
\begin{align*}
&\|\bu\|_{\bL^2(\Omega)}\le C_{poi} \|\bu\|_{\bV},\quad 
\|c\|_{L^2(\Omega)}\le C_{poi} \|c\|_{\tilde{H}^1(\Omega)},\quad
\|\varphi\|_{L^1(\partial\Omega)}\le C_{tr} \|\varphi\|_{W^{1,1}(\Omega)},
\\
&|b_0(\bu,\bv,\bw)|\le C_1 \|\bu\|_{\bV}\|\bv\|_{\bV}\|\bw\|_{\bV},\quad 
|b(\bu,c,n)|\le C_1 \|\bu\|_{\bV}\|c\|_{\tilde{H}^1(\Omega)}\|n\|_{\tilde{H}^1(\Omega)},
\end{align*}
for all $\bu,\bv,\bw\in \bV$, $c,n\in \tilde{H}^1(\Omega)$ and $\varphi\in W^{1,1}(\Omega)$.
For details on Poincar\'e and trace inequalities we refer to~\cite{boyer_book} and 
for the estimates of  $b_0$ and $b$ consult~\cite{temam_1977}.

The main result of the paper is the existence and
uniqueness of weak solutions for
\eqref{eq:navstokes}-\eqref{eq:fronterasup}. Indeed,
let us introduce some appropriate  notation
\begin{align}
&\Theta_1:=\frac{1-C_{tr}}{1-C_{tr}-2\chi\|r\|_{L^1(\mathbb{R})}C_{tr}C_{poi}},
 \quad \Theta_2:=\frac{1-C_{tr}}{1-C_{tr}-C_{tr}C_{poi}},
\label{eq:tetas}
\\
&
\Gamma_0= \frac{|\Omega|\Theta_1C_{poi}}
{|\Omega| -\chi \beta \alpha_1\|r\|^2_{L^\infty(\mathbb{R})}
 C^2_{poi}\Theta_1\Theta_2}
 \left[
  \frac{\chi \alpha_1 \|r\|^2_{L^\infty(\mathbb{R})}\Theta_2}
 {\delta |\Omega|}\|f_c\|_{L^2(\Omega)}+\|f_n\|_{L^2(\Omega)}
 \right],
 \label{eq:gamma0}
 \\
&
\Gamma_1=\frac{\gamma S_c g  C_{poi}}
{S_c-C_1C_{poi}(\gamma g\Gamma_0+\|\bF\|_{\bL^2(\Omega)})},
 \quad 
 \Gamma_2=\frac{1-C_{tr}}{1-2\|r\|_{L^1(\R)}(1-C_{tr}+C_{tr}C_{poi})},
  \label{eq:gamma12}
\\
&
 \Gamma_3=\frac{1-C_{tr}}
 {\delta(1-C_{tr}-C_{tr}C_{poi})-(C_1)^3\|r\|_{Lip(\R)}
 \Gamma_0},
 \label{eq:gamma3}
\end{align}
such that the result is precised as follows:

\begin{theorem}
\label{teo:existence}
Let us consider that  $f_c,f_b\in L^2(\Omega),$
$\mathbf{F}\in \mathbf{L}^2(\Omega)$ and $\overline{n}$, the average of $n$ on $\Omega$,
are given.
Also consider the notation \eqref{eq:tetas}-\eqref{eq:gamma3}.
If we assume that, the following assumptions 
\begin{align}
r\in L^\infty(\mathbb{R})\cap  L^1(\mathbb{R}),\quad 
 1-C_{tr}>C_{tr}C_{poi}\max\{2\chi\|r\|_{L^1(\mathbb{R})},1\},\quad 
 1>\chi \beta \overline{n}\|r\|^2_{L^\infty(\mathbb{R})}
 C^2_{poi}\Theta_1\Theta_2,
 \label{eq:hip_exist}
\end{align}
are satisfied, there is 
$(\bu,p,n,c)\in \bV\times H^1(\Omega)
\times H^1(\Omega)\times H^1(\Omega) $ satisfying
(\ref{eq:navstokes})-(\ref{eq:fronterasup}). Moreover,
if we consider that additionally $r\in {\rm Lip}(\mathbb{R})$
and the following inequalities 
\begin{align}
&S_c-C_1C_{poi}(\gamma g\Gamma_0+\|\bF\|_{\bL^2(\Omega)})>0,
\quad
\delta (1-C_{tr}-C_{tr}C_{poi})-(C_1)^3
\|r\|_{L^1(\mathbb{R})}\Gamma_0>0,
\label{eq:hip_unique_1}
\\
&C_1 \|r\|_{{\rm Lip}(\mathbb{R})}\Gamma_0<1,
\;\;\;
\Pi=\Gamma_1\Gamma_2\left\{
C_1\Gamma_0+
\frac{\|r\|_{L^\infty(\mathbb{R})}C\Gamma_3\Theta_2C_{poi}}
{\delta\big(1-C_1\|r\|_{{\rm Lip}(\mathbb{R})}\Gamma_0\big)}
\Big[
\beta C_{poi}\|r\|_{L^\infty(\mathbb{R})}\Gamma_0+
\|f_c\|_{L^2(\Omega)}
\Big]
\right\}<1,
\label{eq:hip_unique_2}
\end{align}
are satisfied, the weak solution is unique.
\end{theorem}

A similar results are derived in \cite{boldrini,capatina1997} in the case of bioconvection problem 
when the concentration of oxygen is assumed to be constant. In the case of \cite{boldrini} the 
proof is based on the application of Galerkin approximation and in \cite{capatina1997}
on the application of  Gossez theorem. Moreover, other related results are given in 
\cite{liu_2011,kanon}. In particular, in \cite{liu_2011} a well detailed discussion of some particular
models derived from \eqref{eq:navstokes}-\eqref{eq:fronterasup} is given.

\section{Proof of  Theorem~\ref{teo:existence}}

\subsection{Variational formulation}

By the standard arguments the variational formulation of 
\eqref{eq:navstokes}-\eqref{eq:fronterasup} is given by
\begin{align}
&
\mbox{Find $(\bu,n,c)\in  \bV\times  H^1(\Omega)\times H^1(\Omega)$ such that}
\label{eq:weak_form_spaces}
\\
&S_ca_0(\bu,\bv)+b_0(\bu,\bu,\bv)=\gamma S_c (n \mathbf{g}, \bv )
+(\mathbf{F},\bv),\quad \forall\bv\in \bV,
\label{eq:weak_form_ns}\\
&a(n,\phi)+b(\bu,n,\phi)=\chi (n r(c)\nabla c, \nabla \phi )
+(f_n,\phi),
\quad \forall\phi\in H^1(\Omega),
\label{eq:weak_form_bact}\\
&\delta a(c,\varphi)+b(\bu,c,\varphi) 
=-\beta (r(c)n,\varphi)
+\delta \int_{\partial\Omega_U}  \nabla c\cdot \boldsymbol{\nu}\varphi dS
+(f_c,\varphi),
\quad \forall\varphi\in H^1(\Omega).
\label{eq:weak_form_oxi}
\end{align}
We notice that if $f_c=f_n=0$ and $\bu_0$ 
is a solution of  \eqref{eq:navstokes}-\eqref{eq:diver}
with $n=0,$
we have that $(\bu_0,0,0)$ is a solution of 
\eqref{eq:weak_form_spaces}-\eqref{eq:weak_form_oxi}.
However, $(\bu_0,0,0)$ does not describe the bioconvective flow problem and 
we need to study the variational problem when 
the total local concentration of bacteria and the
total local concentration of oxygen are some given  
strictly positive constants, i. e.
$\int_{\Omega}n_{\balpha}d\bx=\alpha_1>0$
and $ \int_{\Omega}c_{\balpha}d\bx=\alpha_2>0.$
Thus, by considering the change of variable
$\hat{n}_{\balpha}=n_{\balpha}-\alpha_1|\Omega|^{-1}$
and 
$\hat{c}_{\balpha}=c_{\balpha}-\alpha_2|\Omega|^{-1},$
we can rewrite \eqref{eq:weak_form_spaces}-\eqref{eq:weak_form_oxi}
as follows
\begin{align}
&\mbox{Given $\balpha=(\alpha_2,\alpha_2)\in ]0,1]\times ]0,1]$
find $(\bu_{\balpha},\hat{n}_{\balpha},\hat{c}_{\balpha})\in 
\bV\times  \tilde{H}^1(\Omega)\times \tilde{H}^1(\Omega):$}
\label{eq:weak_form_spaces:alf:1}
\\
&S_ca_0(\bu_{\balpha},\bv)
+b_0(\bu_{\balpha},\bu_{\balpha},\bv)
=\gamma S_c (\hat{n}_{\balpha} \mathbf{g}, \bv )
+(\mathbf{F},\bv),
\label{eq:weak_form_stokes:alf}
\\
&a(\hat{n}_{\balpha},\phi)+b(\bu_{\balpha},\hat{n}_{\balpha},\phi)=
\chi \left(\Big(\hat{n}_{\balpha}+\frac{\alpha_1}{|\Omega|}\Big)
r\Big(\hat{c}_{\balpha}+\frac{\alpha_2}{|\Omega|}\Big)
\nabla \hat{c}_{\balpha}, \nabla \phi \right)+(f_n,\phi),
\label{eq:weak_form_bacterias:alf}\\
&\delta a(\hat{c}_{\balpha},\varphi)+b(\bu_{\balpha},\hat{c}_{\balpha},\varphi) 
=-\beta \left(r\Big(\hat{c}_{\balpha}+\frac{\alpha_2}{|\Omega|}\Big)
\Big(\hat{n}_{\balpha}+\frac{\alpha_1}{|\Omega|}\Big),\varphi\right)
+\delta \int_{\partial\Omega_U}
\nabla \hat{c}_{\balpha}\cdot \boldsymbol{\nu}\varphi dS
+(f_c,\varphi),
\label{eq:weak_form_oxigeno:alf}
\\
&\mbox{for all $(\bv,\phi,\varphi)\in 
\bV\times  \tilde{H}^1(\Omega)\times \tilde{H}^1(\Omega).$}
\label{eq:weak_form_spaces:alf:2}
\end{align}

\subsection{Some a priori estimates for  $\bu_{\balpha},\hat{n}_{\balpha}$
and $\hat{c}_{\balpha}$}

\begin{proposition}
 \label{prop:dprioriestimates} 
Consider that the hypotheses for existence result 
in Theorem~\ref{teo:existence} are satisfied. 
If we assume that  $(\bu_{\balpha},\hat{n}_{\balpha},\hat{c}_{\balpha})$ is a solution of 
\eqref{eq:weak_form_spaces:alf:1}-\eqref{eq:weak_form_spaces:alf:2},
then  $\|\hat{n}_{\balpha}\|_{\tilde{H}^1(\Omega)} \le \Gamma_0$
with $\Gamma_0$ defined on \eqref{eq:gamma0}
and also are valid the following estimates 
\begin{align}
\|\bu_{\balpha}\|_{\bV} \le C_{poi}\Big(\gamma g \Gamma_0 + \|\bF\|_{\bL^2(\Omega)}\Big),
\qquad 
\|\hat{c}_{\balpha}\|_{\tilde{H}^1(\Omega)} \le
\frac{\Theta_2 C_{poi}}{\delta}\Big[\beta C_{poi}
\|r\|_{L^\infty(\mathbb{R})}
\Gamma_0
+\|f_c\|_{L^2(\Omega)}
\Big].
\label{eq:estima_u_c}
\end{align}
\end{proposition}

\begin{proof}
In order to prove the estimates, we select 
the test functions $(\bv,\phi,\varphi)=(\bu_{\balpha},\hat{n}_{\balpha},\hat{c}_{\balpha})$
in \eqref{eq:weak_form_stokes:alf}-\eqref{eq:weak_form_oxigeno:alf}.
From  \eqref{eq:weak_form_stokes:alf} and \eqref{eq:prop_B0_B1:2} we deduce that
\begin{align}
&\|\bu_{\balpha}\|_{\bV}
\le \gamma  g C^2_{poi}\|\hat{n}_{\balpha}\|_{\tilde{H}^1(\Omega)}
+(S_c)^{-1}C_{poi}\|\mathbf{F}\|_{\mathbf{L}^2(\Omega)}.
\label{eq:estimacion:1}
\end{align}
Now, by the trace inequality and integration by parts,
we have that
\begin{align*}
\int_{\partial\Omega}
|\nabla \hat{n}_{\balpha} \cdot \nu \hat{n}_{\balpha}|dS
\le C_{tr}\|\hat{n}_{\balpha}\nabla \hat{n}_{\balpha} \cdot \nu \|_{W^{1,1}(\Omega)}
\le C_{tr}C_{poi}\|\hat{n}_{\balpha}\|^2_{\tilde{H}^1(\Omega)}
+C_{tr} \int_{\partial\Omega}
|\nabla \hat{n}_{\balpha} \cdot \nu \hat{n}_{\balpha}|dS,
\end{align*}
which implies that 
\begin{align}
\int_{\partial\Omega}
|\nabla \hat{n}_{\balpha} \cdot \nu \hat{n}_{\balpha}|dS
\le \frac{C_{tr}C_{poi}}{1-C_{tr}}\|\hat{n}_{\balpha}\|^2_{\tilde{H}^1(\Omega)}.
\label{eq:trace_estimation}
\end{align}
Here, we have used the fact that $1-C_{tr}>0,$ by a consequence of the
assumption~\eqref{eq:hip_exist}.
Then, by integration by parts we get the following bound 
\begin{align}
&\left(\hat{n}_{\balpha}
r\Big(\hat{c}_{\balpha}+\frac{\alpha_2}{|\Omega|}\Big)
\nabla \hat{c}_{\balpha}, \nabla \hat{n}_{\balpha}\right)
=
\left(
\nabla \left[\int_0^{\hat{c}_{\balpha}} r\Big(m+\frac{\alpha_2}{|\Omega|}\Big)dm\right]
, \nabla \Big(\frac{\hat{n}^2_{\balpha}}{2}\Big)\right)
\nonumber\\
&
\hspace{1cm}
=
-\left(
\int_0^{\hat{c}_{\balpha}} r\Big(m+\frac{\alpha_2}{|\Omega|}\Big)dm
, \Delta \Big(\frac{\hat{n}_{\balpha}}{2}\Big)\right)
+\int_{\partial\Omega}
 \left[\int_0^{\hat{c}_{\balpha}} r\Big(m+\frac{\alpha_2}{|\Omega|}\Big)dm\right]
\nabla \Big(\frac{\hat{n}^2_{\balpha}}{2}\Big)\cdot \nu
dS
\nonumber\\
&
\hspace{1cm}
\le 2\|r\|_{L^1(\mathbb{R})}\int_{\partial\Omega}
|\hat{n}_{\balpha}\nabla \hat{n}_{\balpha} \cdot \nu |dS
\le
\frac{2\|r\|_{L^1(\mathbb{R})}C_{tr}C_{poi}}{1-C_{tr}}
\|\hat{n}_{\balpha}\|^2_{\tilde{H}^1(\Omega)}.
\label{eq:estmacion_int_ene}
\end{align}
From  \eqref{eq:weak_form_bacterias:alf}, using the properties \eqref{eq:prop_B0_B1:2}
and the inequality \eqref{eq:estmacion_int_ene},
we have that
\begin{align*}
&\|\hat{n}_{\balpha}\|^2_{\tilde{H}^1(\Omega)}
=
\chi \left(\hat{n}_{\balpha}
r\Big(\hat{c}_{\balpha}+\frac{\alpha_2}{|\Omega|}\Big)
\nabla \hat{c}_{\balpha}, \nabla \hat{n}_{\balpha}\right)
+
\frac{\chi\alpha_1}{|\Omega|}
\left(
r\Big(\hat{c}_{\balpha}+\frac{\alpha_2}{|\Omega|}\Big)
\nabla \hat{c}_{\balpha}, \nabla \hat{n}_{\balpha}\right)
+(f_n,\phi)
\\
&\;\le
\frac{2\chi\|r\|_{L^1(\mathbb{R})}C_{tr}C_{poi}}{1-C_{tr}}
\|\hat{n}_{\balpha}\|^2_{\tilde{H}^1(\Omega)}
+
\frac{\chi\alpha_1}{|\Omega|}
\|r\|_{L^\infty(\mathbb{R})}
\|\hat{c}_{\balpha}\|_{\tilde{H}^1(\Omega)}
\|\hat{n}_{\balpha}\|_{\tilde{H}^1(\Omega)}
+C_{poi}\|f_n\|_{L^2(\Omega)}\|\hat{n}_{\balpha}\|_{\tilde{H}^1(\Omega)},
\end{align*}
or equivalently, we get the following estimate for $\hat{n}_{\balpha}$
\begin{align}
\|\hat{n}_{\balpha}\|_{\tilde{H}^1(\Omega)}
\le 
\Theta_1
\left[
\frac{\chi\alpha_1}{|\Omega|}
\|r\|_{L^\infty(\mathbb{R})}
\|\hat{c}_{\balpha}\|_{\tilde{H}^1(\Omega)}
+C_{poi}\|f_n\|_{L^2(\Omega)}
\right],
\label{eq:estimacion:2}
\end{align}
with $\Theta_1$ is defined in \eqref{eq:tetas}.
Similarly, from  \eqref{eq:weak_form_oxigeno:alf} 
and \eqref{eq:trace_estimation} with $\hat{c}_{\balpha}$
instead of $\hat{n}_{\balpha}$,
we deduce that
\begin{align}
\|\hat{c}_{\balpha}\|_{\tilde{H}^1(\Omega)}
\le 
\frac{\Theta_2 C_{poi}}{\delta}
\left[\beta C_{poi}
\|r\|_{L^\infty(\mathbb{R})}
\|\hat{n}_{\balpha}\|_{\tilde{H}^1(\Omega)}
+\|f_c\|_{L^2(\Omega)}
\right],
\label{eq:estimacion:3}
\end{align}
where $\Theta_2$ is given in \eqref{eq:tetas}.
Now, 
replacing the estimate \eqref{eq:estimacion:3} 
in \eqref{eq:estimacion:2} and applying~\eqref{eq:hip_exist},
we deduce the existence of $\Gamma_0$ defined in~\eqref{eq:gamma0} 
such that $\|\hat{n}_{\balpha}\|_{\tilde{H}^1(\Omega)} \le \Gamma_0$. 
We notice that the second and third relation in \eqref{eq:hip_exist}
implies that $\Theta_i>1,i=1,2,$
and 
$|\Omega|>\chi \beta \alpha_1\|r\|^2_{L^\infty(\mathbb{R})}
 C^2_{poi}\Theta_1\Theta_2,$ respectively, i.e. $\Gamma>0$
under \eqref{eq:hip_exist}.
Moreover,  from
\eqref{eq:estimacion:1} and \eqref{eq:estimacion:2}, we deduce the estimates
given in \eqref{eq:estima_u_c} for
$\|\bu_{\balpha}\|_{\bV} $
and $\|\hat{c}_{\balpha}\|_{\tilde{H}^1(\Omega)} ,$
concluding the proof of the Proposition.
\end{proof}

\subsection{Proof of Theorem~\ref{teo:existence}}

To prove the existence, we can apply the Gossez theorem \cite{ene_book,Gossez1966}. Indeed, if 
we  define  the mapping 
$G:\bV\times \tilde{H}^1(\Omega)\times \tilde{H}^1(\Omega)\to 
(\bV\times \tilde{H}^1(\Omega)\times \tilde{H}^1(\Omega))'$
by the following relation
\begin{align*}
&<<G(\bu,n,c),(\bv,\phi,\varphi)>>=
\lambda_1\Bigg\{S_ca_0(\bu,\bv)
+b_0(\bu,\bu,\bv)
-\gamma S_c (n \mathbf{g}, \bv )-(\mathbf{F},\bv)\Bigg\}
\\
&
\qquad\qquad
+\lambda_2\Bigg\{a(n,\phi)+b(\bu,n,\phi)
-\chi \left(\Big(n+\frac{\alpha_1}{|\Omega|}\Big)
r\Big(c+\frac{\alpha_2}{|\Omega|}\Big)
\nabla c, \nabla \phi \right)-(f_n,\phi)\Bigg\}
\\
&
\qquad\qquad
+\lambda_3\Bigg\{\delta a(c,\varphi)+b(\bu,c,\varphi) 
+\beta \left(r\Big(c+\frac{\alpha_2}{|\Omega|}\Big)
\Big(n+\frac{\alpha_1}{|\Omega|}\Big),\varphi\right)
-\delta \int_{\partial\Omega_U}
\nabla c\cdot \boldsymbol{\nu}\varphi dS
-(f_c,\varphi)\Bigg\},
\\
&
\hspace{2cm}
\forall (\bu,n,c),(\bv,\phi,\varphi) \in 
\bV\times \tilde{H}^1(\Omega)\times \tilde{H}^1(\Omega),
\end{align*}
with $<<\cdot,\cdot>>$ denoting the duality pairing between  
$\bV\times \tilde{H}^1(\Omega)\times \tilde{H}^1(\Omega)$ and 
$(\bV\times \tilde{H}^1(\Omega)\times \tilde{H}^1(\Omega))'$
and $\lambda_1,\lambda_2$ and $\lambda_3$
are positive fixed constant. From \eqref{eq:prop_B0_B1:2}, \eqref{eq:tetas}
and \eqref{eq:estmacion_int_ene}, we have that
\begin{align*}
&<<G(\bu,n,c),(\bu,n,c)>>\;\;\ge\;\;
\Bigg\{
\lambda_1 S_c\|\bu\|^2_{\bV}
-\lambda_1\gamma S_c g (C_{poi})^2\|n\|_{\tilde{H}^1(\Omega)}\|\bu\|_{\bV}
+\frac{\lambda_2}{3\Theta_1}\|n\|^2_{\tilde{H}^1(\Omega)}\Bigg\}
\\
&
\qquad
+\Bigg\{
\frac{\lambda_2}{3\Theta_1}\|n\|^2_{\tilde{H}^1(\Omega)}
-\frac{\lambda_2\chi\alpha_1}{|\Omega|}
\|r\|_{L^\infty(\mathbb{R})}
\|c\|_{\tilde{H}^1(\Omega)}
\|n\|_{\tilde{H}^1(\Omega)}
+\frac{\lambda_3\delta}{2\Theta_2} \|c\|^2_{\tilde{H}^1(\Omega)}
\Bigg\}
\\
&
\qquad
+\Bigg\{
\frac{\lambda_3\delta}{2\Theta_2} \|c\|^2_{\tilde{H}^1(\Omega)}
-\lambda_3\beta (C_{poi})^2
\|r\|_{L^\infty(\mathbb{R})}
\|c\|_{\tilde{H}^1(\Omega)}
\|n\|_{\tilde{H}^1(\Omega)}
+\frac{\lambda_2}{3\Theta_1}\|n\|^2_{\tilde{H}^1(\Omega)}
\Bigg\}
\\
&
\qquad
-C_{poi}\Bigg\{
\lambda_1\|\mathbf{F}\|_{\bL^2(\Omega)}\|\bu\|_{\bV}
+\lambda_2\|f_n\|_{L^2(\Omega)}\|n\|_{\tilde{H}^1(\Omega)}
+\lambda_3\|f_c\|_{L^2(\Omega)}\|c\|_{\tilde{H}^1(\Omega)}
\Bigg\}:=\Upsilon_1+\Upsilon_2-\Upsilon_3.
\end{align*}
Now, selecting $\lambda_1,\lambda_2,\lambda_3$ and 
$r$ such that
\begin{align*}
&\lambda_1<\frac{4\lambda_2}{3\Theta_1\gamma^2g^2S_c(C_{poi})^4},
&& 
\lambda_2<\frac{4\delta|\Omega|^2\lambda_3}
{6\Theta_1\Theta_2(\chi\alpha_1\|r\|_{L^\infty(\mathbb{R})})^2},
\\
&
\lambda_3<\frac{4\delta\lambda_2}
{6\Theta_1\Theta_2(\beta (C_{poi})^2\|r\|_{L^\infty(\mathbb{R})})^2}
&&
r<\frac{\Upsilon_1+\Upsilon_2}{C_{poi}(\lambda_1\|\mathbf{F}\|_{\bL^2(\Omega)}
+\lambda_2\|f_n\|_{L^2(\Omega)}+\lambda_3\|f_c\|_{L^2(\Omega)})},
\end{align*}
we can prove that $<<G(\bu,n,c),(\bu,n,c)>>$ is positive 
for all $(\bu,n,c)\in \bV\times \tilde{H}^1(\Omega)\times \tilde{H}^1(\Omega)$
such that $\|(\bu,n,c)\|_{\bV\times \tilde{H}^1(\Omega)\times \tilde{H}^1(\Omega)}=r$.
Moreover, we notice that, it is straightforward to deduce that
$G$ is continuous between the weak topologies  of   
$\bV\times \tilde{H}^1(\Omega)\times \tilde{H}^1(\Omega)$ and 
$(\bV\times \tilde{H}^1(\Omega)\times \tilde{H}^1(\Omega))'.$
Thus, there is 
$(\bu,n,c) \in \bar{B}_r(0)
\subset \bV\times \tilde{H}^1(\Omega)\times \tilde{H}^1(\Omega)$
such that $<<G(\bu,n,c),(\bu,n,c)>>= 0$, concluding the proof of existence.

To prove the uniqueness we consider that there is
two solutions $(\bu^i,n^i,c^i),$
$i=1,2$ satisfying
\eqref{eq:weak_form_stokes:alf}-\eqref{eq:weak_form_oxigeno:alf}.
Then, subtracting, selecting the test functions
$(\bv,\phi,\varphi)=(\bu^1-\bu^2,n^1-n^2,c^1-c^2)$,
using \eqref{eq:prop_B0_B1:2}, \eqref{eq:hip_exist}, \eqref{eq:hip_unique_1}
and applying the Proposition~\ref{prop:dprioriestimates}, we get
\begin{align}
&\| \bu^1-\bu^2\|_{\bV}
\le \Gamma_1
\| n^1-n^2\|_{\tilde{H}^1(\Omega)},
\label{eq:es_uniq_uno}\\
&\| n^1-n^2\|_{\tilde{H}^1(\Omega)}
\le 
\Gamma_2
\left[
C_1 \| \bu^1-\bu^2\|_{\bV} 
\|n^1\|_{\tilde{H}^1(\Omega)}
+\|r\|_{L^\infty(\R)}
\| c^1-c^2\|_{\tilde{H}^1(\Omega)}
\right],
\label{eq:es_uniq_dos}\\
&
\| c^1-c^2\|_{\tilde{H}^1(\Omega)}
\le 
C_1\Gamma_3
\left[
 \| \bu^1-\bu^2\|_{\bV} 
\|c^2\|_{\tilde{H}^1(\Omega)}
+(C_1)^2\|r\|_{{\rm Lip}(\R)}
\|n^1\|_{\tilde{H}^1(\Omega)}
\| c^1-c^2\|_{\tilde{H}^1(\Omega)}
\right],
\label{eq:es_uniq_tres}
\end{align}
with $\Gamma_i$ defined on \eqref{eq:gamma0}-\eqref{eq:gamma3}.
From \eqref{eq:es_uniq_tres}, Proposition~\ref{prop:dprioriestimates}
and the first inequality 
in \eqref{eq:hip_unique_2} we have that
\begin{align}
\| c^1-c^2\|_{\tilde{H}^1(\Omega)}
\le 
\frac{C_1\Gamma_3\Theta_2 C_{poi}}{\delta(1-(C_1)^2\|r\|_{{\rm Lip}(\R)}\Gamma_0)}
\Big[\beta C_{poi}
\|r\|_{L^\infty(\mathbb{R})}
\Gamma_0
+\|f_c\|_{L^2(\Omega)}
\Big]
 \| \bu^1-\bu^2\|_{\bV}.
\label{eq:es_uniq_tres_2}
\end{align}
Then, replacing \eqref{eq:es_uniq_tres_2} in  \eqref{eq:es_uniq_dos},
using the Proposition~\ref{prop:dprioriestimates} to 
estimate $\|n^1\|_{\tilde{H}^1(\Omega)}$ we obtain the following bound
$\| n^1-n^2\|_{\tilde{H}^1(\Omega)}\le \Pi(\Gamma_1)^{-1}\| \bu^1-\bu^2\|_{\bV}$
with $\Pi$ defined on \eqref{eq:hip_unique_2}.
Now, using this estimate in \eqref{eq:es_uniq_uno}, we get that 
$\| \bu^1-\bu^2\|_{\bV}\le \Pi \| \bu^1-\bu^2\|_{\bV}$. Thus using the fact that
$\Pi\le 1$ we deduce that $\bu^1=\bu^2$ on $\bV$, 
which also implies that $n^1=n^2$ and $c^1=c^2$ on $\tilde{H}^1(\Omega)$,
concluding the uniqueness proof.

\section*{Acknowledgments}

The authors thanks for the  support of project DIUBB  172409 GI/C at
Universidad del B{\'\i}o-B{\'\i}o, Chile.
AT and IH thanks thanks for the  support of
 Conicyt-Chile through the grants program ``Becas de Doctorado''.

\end{document}